\documentclass[12pt]{amsart}

\usepackage{amssymb, amsfonts, amsmath, mathrsfs, amsthm, fullpage, eucal}

\newtheorem{theorem}{Theorem}[section]

\newtheorem{lemma}[theorem]{Lemma}

\begin{document}


%
%

\title[Moments of the Riemann zeta-function]
{Moments of the Riemann zeta-function at its \\ relative extrema on the critical line}

\author{Micah B. Milinovich}

\address{Department of Mathematics \\ 329 Hume Hall \\ University of Mississippi \\ University, MS 38677 USA} 
\email{mbmilino@olemiss.edu}




\maketitle


\begin{abstract}
Assuming the Riemann hypothesis, we obtain upper and lower bounds for moments of the Riemann zeta-function averaged over the extreme values between its zeros on the critical line. Our bounds are very nearly the same order of magnitude. The proof requires upper and lower bounds for continuous moments of derivatives of Riemann zeta-function on the critical line. 
\end{abstract}


\section{Introduction}

Let $\zeta(s)$ denote the Riemann zeta-function. This article is concerned with estimating discrete moments of the form
\begin{equation*}
\mathcal{M}_k(T)=\frac{1}{N(T)} \sum_{0<\gamma \leq T} \max_{\gamma\leq \tau_\gamma \leq \gamma^+} \big| \zeta(\tfrac{1}{2}\!+\!i \tau_\gamma) \big|^{2k}
\end{equation*}
where $k$ is a natural number and $\gamma \leq \gamma^+$ are successive ordinates of non-trivial zeros of $\zeta(s)$. As usual, the function
\begin{equation}\label{N(T)}
N(T) \ \! := \sum_{0<\gamma\leq T} \!1 \ \! =  \ \! \frac{T}{2\pi} \log  \frac{T}{2\pi} -  \frac{T}{2\pi} + O\big(\log T \big)
\end{equation}
denotes the number of zeros $\rho= \beta+ i \gamma$ of $\zeta(s)$ up to height $T$, counted with multiplicity.  

There have been a number of results concerning behavior of the moments $\mathcal{M}_k(T)$, and other closely related sums, for small values of $k$. See, for instance, \cite{C,CG,H3,H1,H4,L,SS,S,Y}. In this article, we prove the following theorem which provides upper and lower bounds for $\mathcal{M}_k(T)$, for every positive integer $k$, which are very nearly the same order of magnitude. 

\begin{theorem}\label{th}
 Assume the Riemann hypothesis. Let $k$ be a natural number and let $\varepsilon>0$ be arbitrary. Then, letting $\gamma$ and $\gamma^+$ denote consecutive ordinates of non-trivial zeros of $\zeta(s)$, we have
\begin{equation*} 
 (\log T)^{k^2-\varepsilon} \ll \ \! \frac{1}{N(T)} \sum_{0<\gamma \leq T} \max_{\gamma\leq \tau_\gamma \leq \gamma^+} \big| \zeta(\tfrac{1}{2}\!+\!i \tau_\gamma) \big|^{2k}\ \! \ll (\log T)^{k^2+\varepsilon}
\end{equation*}
for sufficiently large $T$ where the implied constants depend on $k$ and $\varepsilon.$
\end{theorem}

With considerably more work, it might be possible to remove the $\varepsilon$ appearing in the exponent of $\log T$ in the lower bound for $\mathcal{M}_k(T)$ using a method of Rudnick and Soundarajan \cite{RS,RS2}. This would be similar to the situation in \cite{MN}, where lower bounds moments of the derivative of the Riemann zeta-function averaged over the non-trivial zeros of $\zeta(s)$ are given. 

Assuming the Riemann hypothesis (RH), Conrey and Ghosh \cite{CG} have shown that
\begin{equation} \label{2ndmoment}
 \mathcal{M}_1(T) \ \! = \ \! \frac{e^2\!-\!5}{2} \log T + O(1).
 \end{equation}
Also, assuming RH, Conrey \cite{C} established the inequalities
$$ \frac{\sqrt{21}}{45\pi} (\log T)^4 \ \! \lesssim \  \! \mathcal{M}_2(T) \! \ \lesssim \ \! \frac{1}{ \pi \sqrt{15}} (\log T)^4.$$
Here, we have used the notation $A\lesssim B$ to indicate that $A \leq \big(1+o(1)\big) B$.
Slightly stronger upper bounds for $\mathcal{M}_2(T)$ have been given by Hall \cite{H1,H4}.

It seems reasonable to conjecture that, for each natural number $k$, 
\begin{equation}\label{conj}
 \mathcal{M}_k(T) \ \!  \asymp \ \! (\log T)^{k^2} \quad \text{ as } \ T\rightarrow \infty.
 \end{equation}
 Theorem \ref{th} certainly lends strong support for this conjecture. In fact, the estimates in (\ref{conj}) could be established if one knew enough about the behavior of continuous moment of $\zeta(s)$ and its derivative on the critical line (see the remark at the end of \textsection 3). It may actually be the case that
$ \mathcal{M}_k(T) \ \! \sim \ \! A_k (\log T)^{k^2},$ when $k\in\mathbb{N}$. However, there does not seem to be a reasonable conjecture for the values of $A_k$ when $k\geq 2.$

\bigskip


\section{Initial Manipulations}

The functional equation for the Riemann zeta-function, in asymmetric form, is 
\begin{equation}\label{fe}
 \zeta(s) = \chi(s) \zeta(1\!-\!s)  \quad \text{where} \quad \chi(s) = 2^s \pi^{s-1} \sin\big(\tfrac{\pi}{2} s\big) \Gamma(1\!-\!s).
 \end{equation}
Note that $|\chi(\tfrac{1}{2}\!+\!it)|=1$ for real $t$ and that $\chi(s)=\chi(1\!-\!s)^{-1}$. Setting
\begin{equation}\label{DefZ}
 Z(t) = \chi(\tfrac{1}{2}\!-\!it)^{\frac{1}{2}} \zeta(\tfrac{1}{2}\!+\!it), 
 \end{equation}
it follows from the functional equation that $Z(t)$ is real for real $t$, that $|Z(t)|=|\zeta(\tfrac{1}{2}\!+\!it)|$, and that $Z$ changes sign when $t$ is an ordinate of a zero of $\zeta(s)$ on the critical line of odd multiplicity. 

If $\gamma$ and $\gamma^+$ denote consecutive ordinates of non-trivial zeros of $\zeta(s)$ on the critical line, then $Z(\gamma)=Z(\gamma^+)=0$ and by Rolle's Theorem there is a zero $\lambda$ of $Z'(t)$ satisfying $\gamma\leq \lambda \leq \gamma^+$.  In fact, on RH, it is known that there is exactly one zero $\lambda$ of $Z'(t)$ satisfying $\gamma\leq \lambda \leq \gamma^+$ when $\gamma>0$; see Theorem 2 of Hall \cite{H2}. Therefore, assuming RH, we see that
$$ \max_{\gamma\leq \tau_\gamma \leq \gamma^+} \big| \zeta(\tfrac{1}{2}\!+\!i\tau_\gamma)\big|^{2k} = Z(\lambda)^{2k}$$
for each $\gamma >0$ and each natural number $k$. Since there are exactly two zeros of $Z'(t)$ between $t=0$ and the smallest positive zero of $Z(t)$, assuming RH it follows that

\begin{equation}\label{2:1}
 \sum_{0<\gamma \leq T} \max_{\gamma\leq \tau_\gamma \leq \gamma^+} \big| \zeta(\tfrac{1}{2}\!+\!i \tau_\gamma) \big|^{2k} = \sum_{\substack{0<\lambda \leq T \\ Z'\!(\lambda)=0}} Z(\lambda)^{2k} + O_k(1) 
 \end{equation}
for each $k\in\mathbb{N}$. 

The sum over the zeros of $Z'(t)$ appearing on the right-hand side of the above equation can be replaced by an integral involving $Z(t)$ and $Z'(t)$ plus a small error term. Assuming RH, let $\lambda$ denote the unique zero of $Z'(t)$ between two consecutive zeros $\gamma$ and $\gamma^+$ of $Z(t)$. Then
\begin{equation}\label{2:2}
\begin{split}
 \int_\gamma^{\gamma^+} \! \big| Z'(t) Z(t)^{2k-1} \big| \ \! dt \ &= \ \! \bigg| \int_\gamma^{\lambda} \! Z'(t) Z(t)^{2k-1} \ \! dt  -\!  \int_{\lambda}^{\gamma^+} \! Z'(t) Z(t)^{2k-1} \ \! dt  \ \! \bigg| = \ \frac{1}{k} \ \! Z(\lambda)^{2k}
\end{split}
\end{equation}
for every natural number $k$. Thus,  it follows from (\ref{2:1}), (\ref{2:2}), and the Lindel\"{o}f hypothesis (which is a consequence of RH) that\begin{equation}\label{2:3}
  \sum_{0<\gamma \leq T} \max_{\gamma\leq \tau_\gamma \leq \gamma^+} \big| \zeta(\tfrac{1}{2}\!+\!i \tau_\gamma) \big|^{2k} \ \! = \ \! k \int_0^{T} \! \big| Z'(t) Z(t)^{2k-1} \big| \ \! dt  + O_{k,\varepsilon}\big(T^\varepsilon \big) 
  \end{equation}
for each $k\in\mathbb{N}$ and every $\varepsilon>0$.

\bigskip


\section{Auxillary Theorems and the Proof of Theorem \ref{th}}

In order to prove Theorem \ref{th}, we require an upper and a lower bound for the integral appearing on the right-hand side of (\ref{2:3}). These bounds will be deduced as a consequence of the following two theorems. Our proof only requires the estimates in Theorem \ref{th2} in the cases where $\ell$ is either $0$ or $1$.

\begin{theorem}\label{cg}
 Assume the Riemann hypothesis. Let $T$ be large and let $k$ be a natural number. Then
\begin{equation}\label{cgbound}
 \int_0^T Z'(t)^2 Z(t)^{2k-2} \ \! dt \ \! \gg \ \! T(\log T)^{k^2+2}
 \end{equation}
where the implied constant depends on $k$.
\end{theorem}

\begin{theorem}\label{th2}
 Assume the Riemann hypothesis. Let $T$ be large, $k$ be a natural number, $\ell$ be a non-negative integer, and $\varepsilon>0$ be arbitrary. Then 
\begin{equation} \label{zeta}
\int_0^T \big|\zeta^{(\ell)}(\tfrac{1}{2}\!+\!it) \big|^{2k} dt \ \! \ll \ \! T (\log T)^{k(k+2\ell)+\varepsilon}
\end{equation}
and
\begin{equation} \label{Z}
 \int_0^T Z^{(\ell)}(t)^{2k} dt \ \! \ll \ \! T (\log T)^{k(k+2\ell)+\varepsilon}
\end{equation}
where the implied constants in $(\ref{zeta})$ and  $(\ref{Z})$ depend on $k, \ell,$ and $\varepsilon.$ 
\end{theorem}

Theorem \ref{cg} is proved using standard Dirichlet polynomial methods.  The proof of Theorem \ref{th2} requires upper bounds for moments of $\zeta(s)$ off of the critical line. We deduce these from a recent result of Soundararajan \cite{S} and a convexity estimate for moments of the Riemann zeta-function due to Heath-Brown \cite{HB}. \\

\begin{proof}[Proof of Theorem \ref{th}] When $k=1$, the theorem follows from the estimate in  (\ref{2ndmoment}). Hence, we may assume that $k\geq 2$ so that $2k-4$ is a non-negative integer.  By H\"{o}lder's inequality and two applications of Theorem \ref{th2} (with $\ell=0$ and $\ell=1$), it follows that
\begin{equation}
\begin{split}
 \int_0^T \big| Z'(t) Z(t)^{2k-1}\big| \ \! dt \ \! &\leq \ \! \left\{ \int_0^T Z(t)^{2k} \ \! dt \right\}^{\frac{2k-1}{2k}} \!\! \cdot \left\{ \int_0^T Z'(t)^{2k} \ \! dt \right\}^{\frac{1}{2k}}
 \\
 & \ll_{_{k,\varepsilon}} T (\log T)^{k^2+1+\varepsilon}.
 \end{split}
 \end{equation}
Combining this estimate with (\ref{N(T)}) and (\ref{2:3}), establishes the upper bound in Theorem \ref{th}.

To establish the lower bound in Theorem \ref{th}, we observe that two more applications of H\"{o}lder's inequality and the bounds in Theorem \ref{th2} (when $\ell=0$ and $\ell=1$) imply that
\begin{equation*}
\begin{split}
 \int_0^T Z'(t)^2 Z(t)^{2k-2} \ \! dt  \ \!  &\leq  \ \!  \left\{ \int_0^T Z'(t)^4 Z(t)^{2k-4} \ \! dt \right\}^{\frac{1}{3}} \!\! \cdot \left\{ \int_0^T \big| Z'(t) Z(t)^{2k-1}\big| \ \! dt  \right\}^{\frac{2}{3}}
 \\
 &\leq  \ \!  \left\{ \int_0^T Z'(t)^{2k} \ \! dt \right\}^{\frac{2}{3k}} \!\! \cdot \left\{ \int_0^T Z(t)^{2k} \ \! dt \right\}^{\frac{k-2}{3k}} \!\! \cdot \left\{ \int_0^T \big| Z'(t) Z(t)^{2k-1}\big| \ \! dt  \right\}^{\frac{2}{3}}\!\!
 \\
 & \ll_{_{k,\varepsilon}} T^{\frac{1}{3}} (\log T)^{\frac{k^2+4}{3}+\frac{2 \varepsilon}{3} }  \cdot \left\{ \int_0^T \big| Z'(t) Z(t)^{2k-1}\big| \ \! dt  \right\}^{\frac{2}{3}}\!\!
\end{split}
\end{equation*}
when $k\geq 2$. Rearranging terms, it follows that
$$ \cdot  \int_0^T \big| Z'(t) Z(t)^{2k-1}\big| \ \! dt  \  \gg_{k,\varepsilon}  \ \frac{ \left\{\int_0^T Z'(t)^2 Z(t)^{2k-2} \ \! dt  \right\}^{\frac{3}{2}}}{T^{\frac{1}{2}} (\log T)^{\frac{k^2+4}{2}+\varepsilon }} . $$
Thus, by (\ref{2:3}) and Theorem \ref{cg},  we find that
\begin{equation*}
\begin{split}
 \sum_{0<\gamma \leq T} \max_{\gamma\leq \tau_\gamma \leq \gamma^+} \big| \zeta(\tfrac{1}{2}\!+\!i \tau_\gamma) \big|^{2k} \ \gg_{k,\varepsilon} \ \! T(\log T)^{k^2+1-\varepsilon}.
 \end{split}
\end{equation*}
Dividing both sides of the above inequality by $N(T)$ and using (\ref{N(T)}), we establish the lower bound in Theorem \ref{th}.
\end{proof}

\noindent{\sc Remarks.}  

\begin{enumerate}
\item In order to establish the estimate for the lower bound for $\mathcal{M}_k(T)$ in the above proof, one can use other exponents in H\"{o}lder's inequality in place of $\frac{2}{3k}, \frac{k-2}{3k},$ and $\frac{2}{3}$. For instance, we could have instead used $\frac{3}{4k}, \frac{2k-3}{4k},$ and $\frac{1}{2}$. Other choices of exponents are possible, as well.  \\

\item If it could be shown that 
\begin{equation*}
\int_0^T \big|\zeta (\tfrac{1}{2}\!+\!it) \big|^{2k} dt \ \! \ll \ \! T (\log T)^{k^2} \quad \text{and} \quad \int_0^T \big|\zeta'(\tfrac{1}{2}\!+\!it) \big|^{2k} dt \ \! \ll \ \! T (\log T)^{k(k+2)}
\end{equation*}
for each natural number $k$, then our proof of Theorem \ref{th} could be modified to show that, assuming RH, 
\begin{equation*}
\begin{split}
 \sum_{0<\gamma \leq T} \max_{\gamma\leq \tau_\gamma \leq \gamma^+} \big| \zeta(\tfrac{1}{2}\!+\!i \tau_\gamma) \big|^{2k}  \asymp \ \! T(\log T)^{k^2+1}.
 \end{split}
\end{equation*}
This would establish the correct order of magnitude for the moments $\mathcal{M}_k(T)$.
\end{enumerate}

\bigskip

\bigskip


\section{Lemmas}

In this section, we collect some estimates which we use to deduce Theorems \ref{cg} and \ref{th2}.

\begin{lemma} \label{L1}
 Assume the Riemann hypothesis. Let $T$ be large, $k>0$ be fixed, and $\varepsilon>0$ be arbitrary. Then
 $$ \int_0^T \big| \zeta(\tfrac{1}{2}\!+\!it) \big|^{2k} \ \! dt \ll_{k,\varepsilon} T (\log T)^{k^2+\varepsilon}.$$
 \end{lemma}
\begin{proof}
This is Corollary A of Soundararajan \cite{S}.
\end{proof}

For $k>0$ and $T\geq 2$, let
$$ J_k(\sigma) = \int_{-\infty}^\infty \big| \zeta(\sigma\!+\!it) \big|^{2k} w_k(t) \ \! dt \quad \text{where} \quad w_k(t) = \int_T^{2T} e^{-2k(t-\tau)^2} d\tau.$$
Then the following estimate holds.
\begin{lemma} \label{L2}
Let $\tfrac{1}{2}\leq \sigma \leq \tfrac{3}{4}, k>0,$ and $T\geq 2$. Then
$$ J_k(\sigma) \ll T^{\sigma-1/2} J_k(\tfrac{1}{2})^{3/2-\sigma}+e^{-kT^2/4}.$$
\end{lemma}
\begin{proof}
This is Lemma 5 of Heath-Brown \cite{HB}. 
\end{proof}

\begin{lemma} \label{L3}
 Assume the Riemann hypothesis. Let $T$ be large, $k>0$ be fixed, and $\varepsilon>0$ be arbitrary. Then
\begin{equation}\label{hijk}
 \int_0^T \big| \zeta(\sigma\!+\!it) \big|^{2k} \ \! dt \ll_{k,\varepsilon} T (\log T)^{k^2+\varepsilon}
 \end{equation}
uniformly for  for $\tfrac{1}{2}\leq \sigma \leq \tfrac{3}{4}$.
\end{lemma}


\begin{proof}
Note that $w_k(t) \gg 1$ for $T\leq t \leq 2T$ and that $$ w_k(t) \ll \exp\left( -(T^2+t^2)k/18\right)$$ for $t\leq 0$ or $t\geq 3T$. Hence, using the estimate $|\zeta(\sigma+it)| \ll_\varepsilon (|t|+1)^{\varepsilon}$ which (on RH) holds uniformly for $\sigma\geq \tfrac{1}{2}$ and $|t|\geq 1$, 
we observe that
$$ J_k(\sigma) \gg \int_T^{2T} \big| \zeta(\sigma\!+\!it) \big|^{2k} \ \! dt \quad \text{and} \quad J_k(\tfrac{1}{2}) \ll 1+ \int_0^{3T} \big| \zeta(\tfrac{1}{2}\!+\!it) \big|^{2k} \ \! dt.$$
Therefore, for $\tfrac{1}{2}\leq \sigma \leq \tfrac{3}{4}$, Lemma \ref{L1} and Lemma \ref{L2} imply that
$$\int_T^{2T} \big| \zeta(\sigma\!+\!it) \big|^{2k} \ \! dt  \ll T (\log T)^{k^2+\varepsilon}$$
The lemma now follows by summing this estimate over the dyadic intervals $[T,T/2],$ $[T/2, T/4],$ $[T/4, T/8], \ldots $ .
\end{proof}

\noindent{\sc Remarks.}

\begin{enumerate}
\item Assuming RH, for each fixed $\sigma >\tfrac{1}{2}$ and $k\in\mathbb{N}$, it is known that
$$ \int_1^T \big| \zeta(\sigma\!+\!it) \big|^{2k} \ \! dt \sim \big(T\!-\!1\big) \sum_{n=1}^\infty \frac{d_k(n)^2}{n^{2\sigma}}$$
where $d_k(n)$ is the number of ways express an integer $n\geq 1$ as the product of exactly $k$ positive integers. This follows, for instance, from Theorem 13.2 of Titchmarsh \cite{T}. Therefore, the above lemma is only non-trivial if $\sigma=\sigma(T)\rightarrow\tfrac{1}{2}$ as $T\rightarrow \infty$. \\

\item Using the argument at the start of the proof of Theorem \ref{th2}, the proof of the above lemma can be altered to show that (\ref{hijk}) holds uniformly for 
$\frac{1}{2}-\frac{c}{\log T} \leq \sigma \leq \frac{3}{4}$ for any fixed constant $c>0$. In this case, the implied constant would depend on $k,\varepsilon,$ and $c$. 
\end{enumerate}

\begin{lemma} \label{chi_estimate}
Let $k$ be a positive integer. Then
\begin{equation*}
\left| \Big(\frac{d}{dt}\Big)^{\! k} \chi(\tfrac{1}{2}\!-\!it)^{\frac{1}{2}}\right| \ll_{_k} \log^k T
\end{equation*}
uniformly for $|t|\leq T $ and $T\geq 2$. Here $\chi(s)$ is the function defined in \textup{(}\ref{fe}\textup{)}.
\end{lemma}

\begin{proof}
This lemma can be established by induction on $k$ using the estimates
\begin{equation} \label{est}
 \frac{\chi'}{\chi}(\sigma\!+\!it) = -\log\frac{|t|}{2\pi} + O\big(|t|^{-1}\big) \quad \text{and} \quad \left| \Big(\frac{d}{dt}\Big)^{\! m} \ \! \frac{\chi'}{\chi}(\sigma\!+\!it)\right| \ll_{_m} |t|^{-1}
 \end{equation}
for $m\in\mathbb{N}$ which, by Stirling's formula, hold uniformly for $1\leq|t|\leq T$ and $-1\leq \sigma \leq 2$. 
\end{proof}

\begin{lemma}\label{cauchy}
 Assume the Riemann hypothesis. Let $k, \ell \in \mathbb{N}$  and let $0<R < \frac{1}{2}$. Then 
\begin{equation*}\label{CE1}
\int_0^T \big|\zeta^{(\ell)}(\tfrac{1}{2}\!+\!it) \big|^{2k} dt \leq \Big(\frac{\ell!}{R^\ell} \Big)^{2k}\cdot \left[ \max_{|\alpha|\leq R} \ \int_0^T \big|\zeta(\tfrac{1}{2}\!+\!\alpha\!+\!it) \big|^{2k} dt \right].
\end{equation*}
\end{lemma}

\begin{proof}
By Cauchy's integral formula,
$$\int_0^T \big|\zeta^{(\ell)}(\tfrac{1}{2}\!+\!it) \big|^{2k} dt = \Big(\frac{\ell!}{2\pi} \Big)^{2k} \cdot \int_0^T \left| \int_{\mathscr{C}_R} \frac{\zeta(\tfrac{1}{2}\!+\!\alpha\!+\! it )}{\alpha^{\ell+1}} \ \! d\alpha \right|^{2k} dt $$
where $\mathscr{C}_R$ denotes the positively oriented circle in the complex plane centered at $0$ of radius $R$. Applying H\"{o}lder's inequality to the integral over $\alpha$ on the left-hand side of the above equation, it follows that
\begin{equation*}
\begin{split}
\int_0^T \big|\zeta^{(\ell)}(\tfrac{1}{2}\!+\!it) \big|^{2k} dt &\leq \Big(\frac{\ell!}{2\pi} \Big)^{2k} \cdot \int_0^T  \left\{ \int_{\mathscr{C}_R} |\alpha|^{\frac{-2k (\ell+1)}{2k-1}} d\alpha \right\}^{2k-1} \cdot \left\{   \int_{\mathscr{C}_R} \big|\zeta(\tfrac{1}{2}\!+\!\alpha\!+\! it )\big|^{2k} \ \! d\alpha  \right\}  \ \! dt
\\
& \leq \Big(\frac{\ell!}{2\pi} \Big)^{2k} \cdot \int_0^T   \frac{(2\pi R)^{2k-1}}{R^{2k(\ell+1)}} \cdot \left\{   \int_{\mathscr{C}_R} \big|\zeta(\tfrac{1}{2}\!+\!\alpha\!+\! it )\big|^{2k} \ \! d\alpha  \right\}  \ \! dt
\\
&\leq  \frac{(\ell!)^{2k}}{2\pi} \frac{1}{R^{2k\ell+1}} \int_0^T  \int_{\mathscr{C}_R} \big|\zeta(\tfrac{1}{2}\!+\!\alpha\!+\! it )\big|^{2k} \ \! d\alpha \ \! dt 
\\
&\leq  \frac{(\ell!)^{2k}}{2\pi} \frac{1}{R^{2k\ell+1}}  \int_{\mathscr{C}_R} \int_0^T  \big|\zeta(\tfrac{1}{2}\!+\!\alpha\!+\! it )\big|^{2k} \ \! dt  \ \! d\alpha 
\\
&\leq \Big(\frac{\ell!}{R^\ell} \Big)^{2k}  \cdot \max_{|\alpha|\leq R} \ \int_0^T \big|\zeta(\tfrac{1}{2}\!+\!\alpha\!+\!it) \big|^{2k} dt
\end{split}
\end{equation*}
as claimed.
\end{proof}


\section{Proof of Theorem \ref{cg}}

If $k$ is 1 or 2, then the integral appearing in Theorem \ref{cg} can be estimated asymptotically using standard methods.  Hence, we will assume throughout the proof that $k\geq 2$ is an integer (so that $k-1$ is also a positive integer). Also, we let $\varepsilon>0$  be an arbitrarily small positive constant which may not be the same at each occurrence. Throughout the proof, implied constants are allowed to depend on $k$ and $\varepsilon$ but are otherwise absolute.  

Following Conrey and Ghosh \cite{CG}, we set 
\begin{equation}\label{Z_1}
Z_1(s)=\zeta'(s)-\frac{1}{2} \frac{\chi'}{\chi}(s) \zeta(s).
\end{equation}
Using (\ref{fe}) and (\ref{DefZ}), one can show that $ |Z'(t)| =  \big| Z_1(\tfrac{1}{2}\!+\!it)\big| $ for every real number $t$; this is assertion (i) of the main lemma in \cite{CG}. Hence
\begin{equation}\label{6:1}
 \int_0^T Z'(t)^2 Z(t)^{2k-2} \ \! dt \ = \ \! \int_0^T \big| Z_1(\tfrac{1}{2}\!+\!it) \zeta(\tfrac{1}{2}\!+\!it)^{k-1} \big|^2 \ \! dt.
\end{equation}
Now let $A(t)$ be a complex valued function which is regular for real $t$. Then, by the Cauchy-Schwarz inequality and (\ref{6:1}), it follows that
\begin{equation}\label{6:2}
\left| \int_0^T  Z_1(\tfrac{1}{2}\!+\!it) \zeta(\tfrac{1}{2}\!+\!it)^{k-1} \overline{A(t)} \ \! dt\right|^2  \leq  \left[ \int_0^T Z'(t)^2 Z(t)^{2k-2} \ \! dt \right] \cdot \left[ \int_0^T \big| A(t) \big|^2 \ \! dt\right].
\end{equation}
We prove Theorem \ref{cg} by obtaining an upper bound for the integral appearing on the right-hand side of the above inequality and a lower bound for the integral appearing on the left-hand side for a specific choice of test function $A(t)$. 

Let $\xi=T^\vartheta$ for a fixed value of $\vartheta$ satisfying $0<\vartheta <1$ and set $A(t)=\mathcal{A}(\frac{1}{2}\!+\!it)$ where
$$ \mathcal{A}(s) = \sum_{n\leq \xi} \frac{d_k(n)}{n^{s}}$$
and $d_k(n)$ is the number of ways express a natural number $n$ as the product of exactly $k$ positive integers. It turns out that the value of $\vartheta$, which we choose at the end of the proof, is important in our argument. Since
$$ \zeta(s)^k = \Big(\sum_{n=1}^\infty \frac{1}{n^s} \Big)^k = \sum_{n=1}^\infty \frac{d_k(n)}{n^s} $$
for $\Re s>1$, we see that $\mathcal{A}(s)$ is an approximation to the $k$th power of the zeta-function. It is known that for each natural number $k$ there exists a constant $C_k$ such that
\begin{equation}\label{tau_k est2}
 \sum_{n\leq \xi} \frac{d_k(n)^2}{n} \ \! = \ \! C_k (\log \xi)^{k^2} + O\Big( (\log \xi )^{k^2-1} \Big)
 \end{equation}
and that
\begin{equation}\label{tau_k est}
\sum_{n\leq \xi} d_k(n)^2  \ \! \ll \ \!  \xi (\log \xi)^{k^2-1}.
\end{equation}
Using these estimates along with Montgomery and Vaughan's mean-value theorem for Dirichlet polynomials (Corollary 3 of \cite{MV}), we see that
\begin{equation}\label{6:3}
\begin{split}
  \int_0^T \big| A(t) \big|^2 \ \! dt  \ \! = \int_0^T \big| \mathcal{A}(\tfrac{1}{2}\!+\! it) \big|^2 \ \! dt\ & = \ \sum_{n\leq \xi} \frac{d_k(n)^2}{n} \big(T + O(n) \big) 
  \\
  & = \ \! C_k \ \! T (\log \xi)^{k^2} + O\Big( T (\log \xi)^{k^2-1} \Big)
 \end{split}
  \end{equation}
  for $2\leq \xi \leq T.$ Thus, it remains to estimate the integral
\begin{equation}\label{6:5}
\begin{split}
 \int_0^T  Z_1(\tfrac{1}{2}\!+\!it) \zeta(\tfrac{1}{2}\!+\!it)^{k-1} \overline{A(t)} \ \! dt &= \frac{1}{i} \int_{1/2}^{1/2+iT} Z_1(s) \zeta(s)^{k-1} \mathcal{A}(1\!-\!s) \ \! ds
 \\
 &= \frac{1}{i} \int_{1/2+i}^{1/2+iT} Z_1(s) \zeta(s)^{k-1} \mathcal{A}(1\!-\!s) \ \! ds + O(1).
\end{split}
\end{equation}

Assuming the Riemann hypothesis, standard estimates for $\zeta(s)$, $\zeta'(s)$, $\chi'(s)/\chi(s)$, and $d_k(n)$ imply that
$$ \big| Z_1(s) \zeta(s)^{k-1} \mathcal{A}(1\!-\!s) \big|  = O\big( T^\varepsilon \xi^{\sigma} \big)$$
uniformly for $\frac{1}{2}\leq \sigma\leq \frac{3}{2}$ and $1\leq t \leq T.$ 
Therefore, by Cauchy's Theorem and (\ref{6:5}), it follows that
\begin{equation*}
\begin{split}
 \int_0^T  Z_1(\tfrac{1}{2}\!+\!it) \zeta(\tfrac{1}{2}\!+\!it)^{k-1} \overline{A(t)} \ \! dt \ &= \ \! \frac{1}{i} \int_{a+i}^{a+iT} Z_1(s) \zeta(s)^{k-1} \mathcal{A}(1\!-\!s) \ \! ds + O\big(T^\varepsilon \xi \big)
 \\
 &= \ \! J_1+J_2+O\big(T^\varepsilon \xi \big),
 \end{split}
 \end{equation*}
with $a=1+(\log T)^{-1}$ and $T\geq 10$, where by (\ref{Z_1}) we have
 $$   J_1 = \frac{1}{i} \int_{a+i}^{a+iT} \zeta'(s) \zeta(s)^{k-1} \mathcal{A}(1\!-\!s) \ \! ds$$
and
 $$     J_2 =  - \frac{1}{2i} \int_{a+i}^{a+iT} \frac{\chi'}{\chi}(s) \zeta(s)^{k} \mathcal{A}(1\!-\!s) \ \! ds.$$
We will show that  $J_1$ is asymptotically negative and that $J_2$ is asymptotically positive. 

In order to estimate $J_1$ and $J_2$, we replace $\zeta'(s)$ and $\zeta(s)$ by their corresponding Dirichlet series. Since these series converge absolutely, we can interchange the sums and the integrals, and then integrate term-by-term. We write
$$ \zeta'(s)\zeta(s)^{k-1}=-\sum_{n=1}^\infty \frac{\tilde{d}_k(n)}{n^s} $$ 
for $\Re s >1$ where 
\begin{equation}\label{tilde}
\tilde{d}_k(n) = \sum_{d|n} d_{k-1}(d) \log\frac{n}{d} \leq d_k(n)  \log n.
 \end{equation}
Then
\begin{equation*}
\begin{split}
J_1 \ \! &= - \sum_{m=1}^\infty \tilde{d}_k(m) \sum_{n\leq \xi} \frac{d_k(n)}{n} \frac{1}{i}  \int_{a+i}^{a+iT} \Big(\frac{m}{n}\Big)^{-s} ds 
\\
&=  - T \ \! \sum_{n\leq \xi} \frac{\tilde{d}_k(n)d_k(n)}{n} + O\Bigg( \sum_{\substack{m\neq n \\ n\leq \xi}} \frac{\tilde{d}_k(m)d_k(n)}{|\log\frac{m}{n}| \ \! m^{a}n^{1-a}} \Bigg).
\end{split}
\end{equation*}
Since
$$ \sum_{m=1}^\infty \frac{d_k(m)^2 \log^2 m}{m^a} \ll (\log T)^{k^2+2},$$
by (\ref{tau_k est}) and (\ref{tilde}) the error term in our estimate for $J_1$ is
\begin{equation*}
\begin{split}
&\ll \ \sum_{m=1}^\infty \frac{d_k(m)^2 \log^2 m}{m^a} \sum_{\substack{n\leq \xi \\ n\neq m}} \frac{1}{|\log\frac{m}{n}|} \ + \ \sum_{n\leq \xi} d_k(n)^2 \sum_{\substack{ m =1 \\ m \neq n}}^\infty \frac{m^{-a}}{|\log\frac{m}{n}|} 
\\
&\ll \ \sum_{m=1}^\infty \frac{d_k(m)^2 \log^2 m}{m^a} \xi \log T \ + \ \sum_{n\leq \xi} d_k(n)^2 \log^2 T
\\
&\ll \ \xi (\log T)^{k^2+3} \ + \ \xi (\log T)^{k^2+1}
\\
&\ll \ T^\varepsilon \xi.
\end{split}
\end{equation*}
Here we have used the inequality $|ab| \leq |a|^2 + |b|^2$ which is valid for $a, b \in \mathbb{C}$.
Hence
$$ J_1 \ = \ - T \ \! \sum_{n\leq \xi} \frac{\tilde{d}_k(n)d_k(n)}{n} + O\big(T^\varepsilon \xi \big).$$

A similar calculation shows that
$$ \frac{1}{i} \int_{a+i}^{a+iT} \zeta(s)^{k} \mathcal{A}(1\!-\!s) \ \! ds  \ = \  T \ \! \sum_{n\leq \xi} \frac{d_k(n)^2}{n} + O\big(T^\varepsilon \xi \big).$$
This expression, (\ref{est}), and (\ref{tau_k est2}) together imply (after an integration by parts) that
$$J_2 \ = \  \frac{T}{2} \log T \ \! \sum_{n\leq \xi} \frac{d_k(n)^2}{n} + O\big(T(\log \xi)^{k^2}\big) + O\big(T^\varepsilon \xi \big).$$
Thus, choosing $\vartheta=\frac{1}{4}$, $\varepsilon=\frac{1}{4}$, and using (\ref{tau_k est2}) and (\ref{tilde}) we see that
\begin{equation}\label{6:4}
\begin{split}
 \int_0^T  Z_1(\tfrac{1}{2}\!+\!it) \zeta(\tfrac{1}{2}\!+\!it)^{k-1} \overline{A(t)} \ \! dt \ &=  \ \! J_1+J_2+O\big(T^\varepsilon \xi \big)
 \\
 &= - T \ \! \sum_{n\leq \xi} \frac{\tilde{d}_k(n)d_k(n)}{n} +  \frac{T}{2} \log T \ \! \sum_{n\leq \xi} \frac{d_k(n)^2}{n} 
 \\
 & \quad \quad\quad + O\big(T(\log \xi)^{k^2}\big) +O\big(T^\varepsilon \xi \big)
 \\
 &\geq - T \ \! \sum_{n\leq \xi} \frac{d_k(n)^2 \log n}{n} +  \frac{T}{2} \log T \ \! \sum_{n\leq \xi} \frac{d_k(n)^2}{n} 
 \\
 & \quad \quad\quad + O\big(T(\log \xi)^{k^2}\big) +O\big(T^\varepsilon \xi \big)
 \\
 & \geq - T \log \xi \ \! \sum_{n\leq \xi} \frac{d_k(n)^2}{n} +  \frac{T}{2} \log T \ \! \sum_{n\leq \xi} \frac{d_k(n)^2}{n} 
 \\
 & \quad \quad\quad + O\big(T(\log \xi)^{k^2}\big) +O\big(T^\varepsilon \xi \big)
 \\
 & \geq \frac{C_k}{4} \ \! T \log T  \ \! (\log \xi)^{k^2} + O\big(T(\log \xi)^{k^2}\big).
 \end{split}
 \end{equation}
Combining the estimates in (\ref{6:2}), (\ref{6:3}), and (\ref{6:4}), we have shown that
\begin{equation*}
\begin{split}
 \int_0^T Z'(t)^2 Z(t)^{2k-2} \ \! dt &\geq  \frac{C_k}{16} T (\log T)^2 (\log \xi)^{k^2} + O\big(T(\log T)^2(\log \xi)^{k^2-1}\big)
 \\
 & \geq \frac{C_k}{4^{k^2+2}} T (\log T)^{k^2+2} + O\big(T(\log T)^{k^2+1}\big).
 \end{split}
 \end{equation*}
This completes the proof of Theorem \ref{cg}.


\section{Proof of Theorem \ref{th2}}

We may assume that $\ell > 0$, as otherwise Theorem \ref{th2} corresponds to Soundararajan's  result stated in Lemma \ref{L1}.  For $\alpha \in \mathbb{C}$, we claim that 
\begin{equation}\label{claim} 
\int_0^T \big| \zeta(\tfrac{1}{2}\!+\!\alpha\!+\!it)\big|^{2k} \ll T (\log T)^{k^2+\varepsilon}
\end{equation}
uniformly for $|\alpha| \leq (\log T)^{-1}$. When $\Re \alpha \geq 0$, this claim is a consequence of Lemma \ref{L3}. When $ \Re \alpha < 0$, Stirling's asymptotic formula for the gamma function (Theorem C.1 of \cite{MV}) 
can be used to show that
\begin{equation*}
\big| \chi(\sigma\!+\!it) \big| = \left(\frac{|t|}{2\pi}\right)^{1/2-\sigma}\Big(1+O\big(|t|^{-1}\big)\Big)
\end{equation*}
uniformly for $-1\leq \sigma\leq 2$ and $|t|\geq 1$ where $\chi(s)$ is defined in (\ref{fe}). Hence
\begin{equation*}
\begin{split}
 \big|\zeta(\sigma\!+\!it)\big| &= \big|\chi(\sigma\!+\!it) \zeta(1\!-\!\sigma\!-\!it) \big| =  \big|\chi(\sigma\!+\!it)| \ \! | \zeta(1\!-\!\sigma\!+\!it) \big| \leq C \big| \zeta(1\!-\!\sigma\!+\!it)  \big|
\end{split}
\end{equation*}
for some absolute constant $C>0$ when $|\sigma-\tfrac{1}{2} |\leq (\log T)^{-1}$ and $0\leq t \leq T$.  Therefore, by another application of Lemma \ref{L3}, it follows that
\begin{equation*}\label{obsv3}
\int_{0}^{T} \big|\zeta(\tfrac{1}{2}\!+\!\alpha\!+\!it)\big|^{2k}  dt  \leq C^{2k}\cdot \int_{0}^{T} \big|\zeta(\tfrac{1}{2}\!-\!\bar{\alpha}\!+\!it)\big|^{2k}  dt \ll T (\log T)^{k^2+\varepsilon}
 \end{equation*}
for $- (\log T)^{-1} \leq \Re \alpha<0$. This proves the claim.

Now let $\ell \geq 1,$ let $k$ be a natural number, and let $R=(\log T)^{-1}.$ Then Lemma \ref{cauchy} and (\ref{claim}) imply that 
\begin{equation}\label{deriv1}
 \int_0^T \big|\zeta^{(\ell)}(\tfrac{1}{2}\!+\!it) \big|^{2k} dt \ \! \ll_{_{k,\ell,\varepsilon}} \ \! T (\log T)^{k(k+2\ell)+\varepsilon}.
 \end{equation}
This proves the first assertion of Theorem \ref{th2}. 

To prove the second assertion, we relate the $\ell$-th derivative of $Z(t)$ for $|t|\leq T$ to a certain linear combination of derivatives of $\zeta(\tfrac{1}{2}\!+\! it)$ and powers of $\log T$, and then use (\ref{deriv1}). Recalling that $|\chi(\tfrac{1}{2}\!+\!it)|=|\chi(\tfrac{1}{2}\!-\!it)|=1$ for real $t$, it follows from (\ref{DefZ}), the Leibniz rule for differentiation, and Lemma \ref{chi_estimate}  that 
\begin{equation*}
\begin{split}
 Z^{(\ell)}(t)  &= \sum_{m=0}^\ell \binom{\ell}{m} \ \!  \Big(\frac{d}{dt}\Big)^m\! \zeta(\tfrac{1}{2}\!+\!it) \ \! \Big(\frac{d}{dt}\Big)^{\ell-m} \!\chi(\tfrac{1}{2}\!-\!it)^{1/2} \ll_{_\ell}  \sum_{m=0}^\ell \big|\zeta^{(m)}(\tfrac{1}{2}\!+\!it)\big| (\log T)^{\ell-m}
 \end{split}
 \end{equation*}
for $|t| \leq T$ when $T$ is sufficiently large.  Thus,
$$ \int_0^T Z^{(\ell)}(t)^{2k} \ \! dt \ \! \ll_{_\ell} \ \!  \sum_{m=0}^\ell \int_0^T \big|\zeta^{(m)}(\tfrac{1}{2}\!+\!it)\big|^{2k} (\log T)^{2k(\ell-m)} \ \! dt \ \! \ll_{_{k,\ell,\varepsilon}} \ \! T (\log T)^{k(k+2\ell)+\varepsilon},$$
as claimed. This completes the proof of Theorem \ref{th2}.


\bigskip

\bigskip

\noindent{\it Acknowledgements.} The author would like to thank Professor D. R. Heath-Brown for some helpful comments which greatly simplified the proof of Theorem \ref{th2}. The author would also like to thank the anonymous referee for a number of useful suggestions.



\end{document}